\documentclass[reqno,12pt]{amsart}
\usepackage{amsmath,amsthm,enumerate,cite,graphics,pstricks,amsfonts,latexsym,amsopn,verbatim,amscd,amssymb}
\usepackage{color}
\usepackage{psfrag}
\usepackage[all]{xy}

\theoremstyle{plain}
\newtheorem{thm}{Theorem}[section]

\newtheorem{lem}[thm]{Lemma}
\newtheorem{cor}[thm]{Corollary}
\newtheorem{prop}[thm]{Proposition}
\newtheorem{conj}[thm]{Conjecture}

\theoremstyle{definition}

\newtheorem{rem}[thm]{Remark}

\newcommand{\A}{\mathcal{A}}
\newcommand{\B}{\mathcal{B}}

\DeclareMathOperator{\Cent}{Cent}

\begin{document} 

\title[Counting centralizers and $z$-classes of some F-groups]{Counting centralizers and $z$-classes of some F-groups} 

\author[S. J. Baishya  ]{Sekhar Jyoti Baishya} 
\address{S. J. Baishya, Department of Mathematics, Pandit Deendayal Upadhyaya Adarsha Mahavidyalaya, Behali, Biswanath-784184, Assam, India.}

\email{sekharnehu@yahoo.com}

\begin{abstract}
A finite group $G$ is called an F-group  if for every $x, y \in G \setminus Z(G)$, $C(x) \leq C(y)$ implies that $C(x) = C(y)$. On the otherhand, two elements of a group are said to be $z$-equivalent or in the same $z$-class if their centralizers are conjugate in the group. In this paper, for a finite non-abelian group, we give  a necessary and sufficient condition for the number of centralizers/ $z$-classes to be equal to the index of its center. We also give a necessary and sufficient condition for the number of $z$-classes of a finite F-group to attain its maximal number (which extends an earlier result). Among other results, we have computed the number of element centralizers and $z$-classes of some  finite groups and extend some previous results.
\end{abstract}

\maketitle

\section{Introduction} \label{S:intro}

In 1953, Ito \cite {ito} introduced the notion of the class of F-groups, consisting of finite groups $G$ in which for every $x, y \in G \setminus Z(G)$, $C(x) \leq C(y)$ implies that $C(x)=C(y)$, where $C(x)$ and $C(y)$ are centralizers of $x$ and $y$ respectively. Since then the influence of the element centralizers on the structure of groups has been studied extensively. An interesting subclass of F-groups is the class of I-groups, consisting of groups in which all centralizers of non-central elements are of same order. Ito in \cite {ito} proved that I-groups are nilpotent and direct product of an abelian group and a group of prime power order. Later on  in 2002, Ishikawa \cite{ish} proved that I-groups are of class at most $3$. In 1971,  Rebmann \cite {reb} investigated and classified F-groups.

In 2007, R. S. Kulkarni \cite{kul, kul1} introduced the notion of $z$-classes in a group. Two elements of a group are said to be $z$-equivalent or in the same $z$-class if their centralizers are conjugate in the group.  $z$-equivalence is an equivalence relation which is weaker than conjugacy relation. An infinite group generally contains infinitely many conjugacy classes, but may have finitely many $z$-classes. In \cite{kul1} the author observed the influence of the $z$-classes in the groups of automorphisms of classical geometries and apart from other results he concludes that this finiteness of $z$-classes can be related to the idea of finiteness of dynamical types of transformation to the geometry. It may be mentioned here that apart from the geometric motivation,  finding $z$-classes of a group  itself is of independent interest as a pure combinatoral problem.  More information on this and related concepts may be found in \cite{zclass, gon1, gon2, gon3, rahul1, rahul3, rahul4}.

In a recent work, the authors in \cite{rahul1} investigated $z$-classes in finite $p$-groups ($p$ a prime). Among other results, they proved that a non-abelian $p$-group $G$ can have at most $\frac{p^k-1}{p-1}+1$ number of $z$-classes, where $\mid \frac{G}{Z(G)}\mid=p^k$ and gave a necessary condition to attain the maximal number which is not sufficient. Recently, the authors in \cite{zclass} gave a necessary and sufficient condition for a finite $p$-group of conjugate type $(n, 1)$ to attain this maximal number. In this paper, apart from other results, we extend this result and give a necessary and sufficient condition for a finite F-group to attain this maximal number. For a finite non-abelian group, we give  a necessary and sufficient condition for the number of centralizers/ $z$-classes to be equal to the index of its center. We disprove a Conjecture in \cite{amiriF1}, namely, if $G$ is a finite  group such that the number of centralizers is equal to the index of its center, then $G$ is an F-group. Among other results, we have computed the number of element centralizers and  $z$-classes of some  groups and improve some earlier results. It may be mentioned here that characterization of groups in terms of the number of element centralizers have been considered by many researchers (see for example \cite{ctc092, zarrin094} for finite groups and \cite{non} for infinite groups).

Throughout this paper $G$ is a  group with center $Z(G)$, commutator subgroup $G'$ and the set of element centralizers $\Cent(G)$. We write $Z(x)$ to denote the center of the proper centralizer $C(x)$ and `$z$-class' to denote the set of $z$-classes in $G$.

 \section{Preliminaries}

 We begin with some  Remarks which will be used in the sequel.

\begin{rem}\label{rem1}
 (See \cite[Pp. 571--572]{zappa}) A collection $\Pi$ of non-trivial subgroups of a finite group $G$ is called a partition if every non-trivial element of $G$ belongs to a unique subgroup in $\Pi$. If $\mid \Pi \mid=1$, the partition is said to be trivial. The subgroups in $\Pi$ are called components of $\Pi$. Following Miller, any finite abelian group having a non-trivial partition is an elementary abelian $p$-group of order $\geq p^2$ ($p$ a prime).

Let $S$ be a subgroup of a finite group $G$. A set $\Pi= \lbrace H_1, H_2, \dots, H_n \rbrace$ of subgroups of $G$ is said to be a strict $S$-partition of $G$ if $S \leq H_i$ ($i=1, 2, \dots, n$) and every element of $G\setminus S$ belongs to one and only one subgroup  $H_i$ ($i=1, 2, \dots, n$). For more information about partition see \cite{part}.

Given a finite group $G$, let $\A= \lbrace C(x) \mid x \in G \setminus Z(G) \rbrace$ and $\B= \lbrace Z(x) \mid x \in G \setminus Z(G) \rbrace$. $\A$ and $\B$ are partially ordered sets with respect to inclusion and they have the same length. The length of $\A$ (and of $\B$) is called the rank of $G$. A finite group $G$ has rank $1$ if and only if $\B$ is a strict $Z(G)$-partition.

Recall that a finite group $G$ is called an F-group  if for every $x, y \in G \setminus Z(G)$, $C(x) \leq C(y)$ implies that $C(x) = C(y)$. Following \cite[Lemma 2.6]{brough}, a finite  group $G$ is an F-group if and only if $\B$ is a strict $Z(G)$-partition. 

Hence being a finite group of rank $1$ is equivalent to being a finite F-group.
\end{rem}

\begin{rem}\label{rem2}
Given a group $G$, two elements $x, y \in G$ are said to be $z$-equivalent or in the same $z$-class if their centralizers are conjugate in $G$, i.e., if $C(x)=gC(y)g^{-1}$ for some $g \in G$.  It is well known that ``being $z$-equivalent" is an equivalence relation on $G$. Following \cite[Theorem 2.1]{kul1}, the order of the $z$-class of $x$, if finite, is given by 
\[
\mid z-class \; of\; x \mid =\mid G : N_G(C(x)) \mid . \mid F_x' \mid,
\]
where $F_x'= \lbrace y \in G \mid C(x)=C(y)\rbrace$. From this it is easy to see that the number of $z$-classes in $G$ i.e., $\mid z-class \mid \leq \mid \Cent(G) \mid$, with equality  if and only if $C(x)\unlhd G$ for all $x \in G $.
\end{rem}

The following Theorems will be used to obtain some of our results. For basic notions of isoclinism, see \cite{hall, rahul1}.

\begin{thm}{(p.135 \cite{hall})} \label{ff1}
Every group is isoclinic to a group whose center is contained in the commutator subgroup. 
\end{thm}

\begin{thm}{(Theorem A \cite{rahul2}, {Lemma 3.2 \cite{non}})} \label{ff2}
Any two isoclinic groups have the same number of centralizers.
\end{thm}

\begin{thm}{(Theorem 11 \cite{rahul2}, Theorem 3.3 \cite {non})} \label{ff3}
The representatives of the families of isoclinic groups with $n$-centralizers ($n \neq 2, 3$) can be chosen to be finite groups.
\end{thm}

\begin{thm}{(Lemma 4, p. 303 \cite{zumud})} \label{ff4}
Let $G$ be a finite non-abelian group with an abelian normal subgroup of prime index $p$.  Then $\mid G \mid=p.\mid Z(G)\mid. \mid G' \mid$.
\end{thm}

\section{The main results}

In this section, we prove the main results of the paper. Let $D_8$ be the dihedral group of order $8$. It is easy to verify that $\mid\Cent(D_8)\mid=\mid \frac{D_8}{Z(D_8)}\mid$. The following result gives two necessary conditions for $\mid \Cent(G) \mid =\mid \frac{G}{Z(G)}\mid$, where $G$ is a finite non-abelian group.

\begin{prop}\label{zclass1}
Let $G$ be a finite non-abelian group such that $\mid \Cent(G) \mid =\mid \frac{G}{Z(G)}\mid$. Then

\begin{enumerate}
	\item  $ \frac{G}{Z(G)}$ is an elementary abelian $2$-group (\cite[Theorem 2.1]{amiriF1}).
	\item $Cl(x) \subseteq xZ(G)$ for any $x \in G$.
\end{enumerate} 
\end{prop}

\begin{proof}
 Suppose $\mid \Cent(G) \mid =\mid \frac{G}{Z(G)}\mid=n$. Then $G$ can be written as: 
\[
G=Z(G) \sqcup x_1Z(G) \sqcup x_2Z(G) \sqcup \dots \sqcup x_{n-1}Z(G).
\]
Consequently, we have $\Cent(G)= \lbrace G, C(x_1), C(x_2), \dots, C(x_{n-1})\rbrace$.

Now, 

a) For any $1\leq i \leq n-1$, the elements whose centralizers equals to $C(x_i)$ are precisely the members of $x_iZ(G)$, noting that  $\mid \Cent(G) \mid =n$ and $C(x_i)=C(x_iz)$ for any $z \in Z(G)$. It now follows that  $a^{-1} \in x_iZ(G)$ for any $a \in x_iZ(G)$, because $C(a)=C(a^{-1})$. Therefore we have $x_iZ(G)= {x_i}^{-1}Z(G)$, which implies $o(x_iZ(G))=2$ (in $\frac{G}{Z(G)}$). Hence $ \frac{G}{Z(G)}$ is an elementary abelian $2$-group.\\

b) In view of part (a), observe that $C(x) \unlhd G$ for any $x \in G$. Consequently, $C(x)=gC(x)g^{-1}=C(gxg^{-1})$ for any $x, g \in G$. Now, using the proof of part (a), we have $gxg^{-1} \in xZ(G)$ for any $x, g \in G$. Hence  $Cl(x) \subseteq xZ(G)$ for any $x \in G$.
\end{proof}

Note that for the group $G:=$ Small group $(64, 60)$ in \cite{gap} (see also \cite{ed09}), we have $\frac{G}{Z(G)} \cong C_2 \times C_2 \times C_2$ and $\mid \Cent(G) \mid =6$.

The authors in \cite[Proposition 2.4]{rahul1} proved that if $G$ is a group in which $Z(G)$ has finite index, then the number of $z$-classes is at most the index $[G:Z(G)]$. It is also easy to see that for such groups the number of element centralizers is at most the index $[G:Z(G)]$ (noting that $C(x)=C(xz)$ for any $z \in Z(G)$). In the following  result, for a finite non-abelian group, we give a necessary and sufficient condition for the number of element centralizers/ $z$-classes to be equal to the index of the center.

\begin{prop}\label{z-class2}
Let $G$ be a finite non-abelian group. Then $\mid \Cent(G) \mid =\mid \frac{G}{Z(G)}\mid$ iff $\mid z-class \mid =\mid \frac{G}{Z(G)}\mid$. 
\end{prop}

\begin{proof}
If $\mid \Cent(G) \mid =\mid \frac{G}{Z(G)}\mid$, then by Proposition \ref{zclass1}, $ \frac{G}{Z(G)}$ is  abelian and consequently, $C(x) \unlhd G$ for any $x \in G$. Now, using Remark \ref{rem2}, we have $\mid z-class \mid =\mid \Cent(G) \mid=\mid \frac{G}{Z(G)}\mid$. 

Conversely, suppose $\mid z-class \mid =\mid \frac{G}{Z(G)}\mid$. In view of Remark \ref{rem2}, we have $\mid z-class \mid \leq \mid \Cent(G) \mid \leq \mid \frac{G}{Z(G)}\mid$ (noting that $C(x)=C(xz)$ for any $z \in Z(G)$). Hence the result follows.
\end{proof}

We now give a counterexample to the following Conjecture in \cite{amiriF1}.

 \begin{conj}\label{conj1}
If $G$ is a finite group such that $\mid \Cent(G) \mid =\mid \frac{G}{Z(G)}\mid$, then $G$ is an F-group.
\end{conj}

Consider the dihedral group $D_8$ of order $8$. We have 
\[
\mid \Cent(D_8 \times D_8)\mid= \mid \Cent(D_8)\mid \times \mid \Cent(D_8)\mid =16= \mid \frac{D_8 \times D_8}{Z(D_8 \times D_8)}\mid.
\]
Let $x \in D_8 \setminus Z(D_8)$ and $z \in Z(D_8)$. Then $(x, x), (x, z) \in D_8 \times D_8$ and 
\[
C((x, x))=C(x) \times C(x) \subsetneq C(x) \times C(z) = C((x, z)).
\]
Hence $D_8 \times D_8$ is not an F-group. \\

However for a finite F-group, we have the following result. Recall that $Z(x)$  denotes the center of the proper centralizer $C(x)$.

\begin{prop}\label{bb1}
Let $G$ be a finite F-group. Then 

\begin{enumerate}
	\item  $\mid \Cent(G) \mid =\mid \frac{G}{Z(G)}\mid$ iff $\mid \frac{Z(x)}{Z(G)}\mid=2$ for all $x \in G \setminus Z(G)$.
	\item $\mid z-class \mid =\mid \frac{G}{Z(G)}\mid$ iff $\mid \frac{Z(x)}{Z(G)}\mid=2$ for all $x \in G \setminus Z(G)$. 
\end{enumerate}

\end{prop}

\begin{proof}
a) 
Suppose $\mid \Cent(G) \mid =\mid \frac{G}{Z(G)}\mid=n$. Then $G$ can be written as: 
\[
G=Z(G) \sqcup x_1Z(G) \sqcup x_2Z(G) \sqcup \dots \sqcup x_{n-1}Z(G).
\]
Consequently, we have $\Cent(G)= \lbrace G, C(x_1), C(x_2), \dots, C(x_{n-1})\rbrace$.

Now, suppose $\mid \frac{Z(x_i)}{Z(G)}\mid> 2$ for some $x_i, 1 \leq i \leq n-1$. Then $x_jZ(G) \in \frac{Z(x_i)}{Z(G)}$ for some $j \neq i$, $1 \leq j \leq n-1$. But then $x_j \in Z(x_i)$ and consequently, $C(x_i) \subseteq C(x_j)$, which is a contradiction since $G$ is an F-group. Hence the result follows.

Conversely, suppose $\mid \frac{Z(x)}{Z(G)}\mid=2$ for all $x \in G \setminus Z(G)$. Then $\mid \Cent(G) \mid =\mid \frac{G}{Z(G)}\mid$, noting that in the present scenario, for any $x \in G \setminus Z(G)$, $Z(x)$ will contain exactly two right cosets of $Z(G)$.\\

b) The result follows from Proposition \ref{z-class2}. 
\end{proof}

For any subgroup $H$ of a group $G$, it is easy to see that $C_H(x)=C_G(x)\cap H$, for any $x \in H$. This gives the following result:

\begin{lem}\label{lemm1}
Let $H$ be a subgroup of a group $G$ such that $Z(G) \lneq Z(H)$. Then the number of  centralizers of $G$ produced by elements of $H$ is at least $\mid \Cent(H)\mid +1$.    
\end{lem}

\begin{proof}
Clearly, the number of   centralizers of $G$ produced by elements of $H$ is equal to the number of centralizers of $G$ produced by the elements of $Z(G)$+ the number of centralizers of $G$ produced by the elements of $H \setminus Z(G) \geq 1+\mid \Cent(H)\mid $ (Note that elements of $H$ that have the same centralizers in $H$ may have different centralizers in $G$).
\end{proof}

The following key Lemma improves \cite[Lemma 2.1 and Corollary 2.2]{jaa4}.

\begin{lem}\label{lemm2}
Let $G$ be a non-abelian group and $x \in G \setminus Z(G)$. Then 
\begin{enumerate}
	\item Suppose $x$ has finite order. Then $\mid \Cent(G)\mid \geq  \mid \Cent(C(x))\mid +p+1$, where $p$ is the smallest prime divisor of $o(x)$.
	\item Suppose $x$ has infinite order. Then $\mid \Cent(G)\mid \geq  \mid \Cent(C(x))\mid +3$.
\end{enumerate}  
\end{lem}

\begin{proof}
a) Let $a \in G \setminus C(x)$. Clearly, $ax^i \in G \setminus C(x)$ for any $i$.
Moreover, we have $gcd(m, o(x))=1$ for $1< m< p$ and consequently, $C(x)=C(x^m)$. It now follows that $C(a), C(ax), C(ax^2), \dots, C(ax^{p-1})$ are all distinct centralizers of $G$ (because, if $C(ax^i)=C(ax^j)$ for some $0 \leq i < j \leq p-1$, then $a \in C(x^{j-i})=C(x)$, which is a contradiction). In the present scenario, if $C(ax^i)=C(y)$ for some $y \in C(x)$ and some $i$, then $a \in C(x)$, which is a contradiction.

Now, $\mid\Cent(G)\mid$ = no. of centralizers of $G$ produced by elements of $C(x)$+ no. of centralizers of $G$ produced by elements of $G \setminus C(x)  \geq 1+ \mid \Cent(C(x)) \mid+p$ by Lemma \ref{lemm1}.\\

b) Using arguments similar to (a) we get the result.
\end{proof}

Recall that a finite group $G$ is said to be a CA-group   if centralizer of every noncentral element of $G$ is abelian.

\begin{cor}\label{lemm25}
Let $G$ be a finite non-abelian  group and $p(q)$ be the smallest (largest) prime divisor of its order. Then 
\begin{enumerate}
	\item $2+p \leq \mid\Cent(G)\mid$. Moreover, if $G$ is not a CA-group, then $5+p \leq \mid\Cent(G)\mid$.
	\item If $G$ has trivial center, $2+q \leq \mid\Cent(G)\mid$. 
\end{enumerate}  
\end{cor}

\begin{proof}
The proof follows from Lemma \ref{lemm2}, noting that in case of (b), $G$ has a non-central element say $x$ of order $q$. On the other hand, if $G$ is not a CA-group, then $C(x)$ will be non-abelian for some $x \in G \setminus Z(G)$ and consequently, $4 \leq \mid\Cent(C(x))\mid$. Hence our result follows by noting that $p$ is the smallest prime divisor of $o(x)$. 
\end{proof}

Recall that a finite $p$-group ($p$ a prime) $G$ is said to be a special $p$-group of rank $k$ if $G'=Z(G)$ is elementary abelian of order $p^k$ and $\frac{G}{G'}$ is elementary abelian. Furthermore, a finite group $G$ is extraspecial if $G$ is a special $p$-group and $\mid G' \mid=\mid Z(G) \mid=p$. We now give the following result concerning the upper and lower bounds of $\mid \Cent(G) \mid$. Note that part (a) of the following Theorem improves  \cite[Lemma 2.7]{en09}.

\begin{thm}\label{bs1}
Let $G$ be a finite non-abelian group and $p$ be the smallest prime divisor of its order. Then

\begin{enumerate}
	\item $p+2 \leq \mid \Cent(G) \mid$; with equality if and only if $\frac{G}{Z(G)} \cong C_p \times C_p$; equivalently, if and only if $G$ is isoclinic to an extraspecial group of order $p^3$.
	\item  $\mid \Cent(G) \mid \leq \mid \frac{G}{Z(G)}\mid$; with equality if and only if $\mid z-class \mid =\mid \frac{G}{Z(G)}\mid$; equivalently, $G=H \times A$ where $A$ is an abelian subgroup of odd order and $H$ is a Sylow $2$-subgroup of $G$ such that $\frac{H}{Z(H)}$ is elementary abelian;  equivalently, $G$ is isoclinic to a special $2$-group.
\end{enumerate} 
\end{thm}

\begin{proof}
a) Using Corollary \ref{lemm25}, we have $p+2 \leq \mid \Cent(G) \mid$.

Now, suppose $p+2 = \mid \Cent(G) \mid$. Then  by Corollary \ref{lemm25}, $G$ is a CA-group.
Moreover, in view of \cite[Lemma 3.3]{tom}, $\mid \frac{G}{C(x)}\mid=p$ for any $x \in G \setminus Z(G)$. Therefore $G=\langle C(x), C(y) \rangle$ for some $x, y \in G \setminus Z(G)$ and consequently, by \cite[Theorem 4.2]{tom}, $\frac{G}{Z(G)} \cong C_p \times C_p$. Conversely, suppose $\frac{G}{Z(G)} \cong C_p \times C_p$. Then $p+2 = \mid \Cent(G) \mid$, noting that in the present scenario $C(x) \cap C(y)=Z(G)$ for any $x, y \in G \setminus Z(G)$ with $xy \neq yx$ and each proper centralizer of $G$ contains exactly $p$ distinct right cosets of $Z(G)$.

For the last part, suppose $G$ is isoclinic to an extraspecial  group $N$ of order $p^3$. Then $\frac{N}{Z(N)} \cong C_p \times C_p$ and so $ \mid \Cent(N) \mid=p+2$. Now, the result follows using Theorem \ref{ff2}.

Conversely, suppose $p+2 = \mid \Cent(G) \mid$. Then we have $\frac{G}{Z(G)} \cong C_p \times C_p$. In the present scenario, by Theorem \ref{ff1} and Theorem \ref{ff3}, $G$ is isoclinic to a finite group $N$ of order $p^n$ with $Z(N) \subseteq N'$. Moreover, since $\frac{N}{Z(N)} \cong C_p \times C_p$, therefore $Z(N) = N'$. Also note that any proper centralizer of $N$ is abelian normal of index $p$ in $N$. Hence using Theorem \ref{ff4}, we have $\mid N \mid=p.\mid Z(N) \mid .\mid N'\mid$ and consequently, $N$ is an extraspecial
group of order $p^3$.\\

b) Since $C(x)=C(xz)$ for any $z \in Z(G)$, therefore  $\mid \Cent(G) \mid \leq \mid \frac{G}{Z(G)}\mid$; with equality if and only if $\mid z-class \mid =\mid \frac{G}{Z(G)}\mid$  by Proposition \ref{z-class2}. Last part follows from \cite[Proposition 2.5]{rahul1}.

\end{proof}

We now give our first counting formula for number of distinct centralizers:

\begin{prop}\label{f1}
Let $G$ be a finite F-group such that $\mid \frac{G}{Z(G)}\mid=p^k$ ($p$ a prime). If $\mid \frac{Z(x)}{Z(G)}\mid=p^m$ for all $x \in G \setminus Z(G)$, then $\mid \Cent(G) \mid =\frac{p^k-1}{p^m-1}+1$.
\end{prop}

\begin{proof}
Since $G$ is a finite F-group, therefore by Remark \ref{rem1}, we have  $Z(x) \cap Z(y)=Z(G)$ for any $x, y \in G \setminus Z(G)$ with $C(x) \neq C(y)$. Hence the result follows by noting that for any $x \in G \setminus Z(G)$, $Z(x)$ contains exactly $p^m$ distinct right cosets of $Z(G)$.
\end{proof}

The following result generalizes \cite[Theorem 3.5]{amiriF}.

\begin{thm}\label{b1}
Let $G$ be a finite group and $p$ a prime. Then $\mid \frac{Z(x)}{Z(G)}\mid=p$ for all $x \in G \setminus Z(G)$ if and only if
\begin{enumerate}
	\item $G$ is an F-group.
	\item $\frac{G}{Z(G)}$ is of exponent $p$.
	\item  $\mid \Cent(G) \mid =\frac{p^k-1}{p-1}+1$, where $\mid \frac{G}{Z(G)}\mid=p^k$.
\end{enumerate} 
\end{thm}

\begin{proof}

Suppose $\mid \frac{Z(x)}{Z(G)}\mid=p$ for all $x \in G \setminus Z(G)$.\\
a) Clearly, $\Pi=\lbrace \frac{Z(x)}{Z(G)} \mid x \in G \setminus Z(G)\rbrace$ is a partition of $\frac{G}{Z(G)}$. Hence  by Remark \ref{rem1}, $G$ is an F-group.\\

b) It is clear from the proof of (a) that every element of $\frac{G}{Z(G)}$ is of order $\leq p$.\\

c) Clearly, for any $x \in G \setminus Z(G)$, $Z(x)$ contains exactly $p$ distinct right cosets of $Z(G)$. Hence the result follows.

Conversely, suppose $G$ is an F-group,  $\mid \frac{G}{Z(G)}\mid=p^k$ and  $\mid \Cent(G) \mid =\frac{p^k-1}{p-1}+1 (=l)$. Let $C(x_i),1 \leq i \leq l-1 $ be the proper centralizers of $G$. Since $G$ is an F-group, therefore by Remark \ref{rem1}, $\Pi= \lbrace \frac{Z(x_i)}{Z(G)} \mid 1\leq i \leq  l-1\rbrace$ is a partition of $\frac{G}{Z(G)}$. In the present scenario, we have  $p^k=\mid \frac{Z(x_1)}{Z(G)} \mid+ \mid \frac{Z(x_2)}{Z(G)} \mid+ \dots+\mid \frac{Z(x_{l-1})}{Z(G)} \mid -\frac{p^k-1}{p-1}+1$. Consequently, $\mid \frac{Z(x)}{Z(G)} \mid=p$ for all $x \in G \setminus Z(G)$.
\end{proof}

As an immediate corollary, we have the following result.

\begin{cor}\label{cor1}
Let $G$ be a finite F-group such that $\mid \frac{G}{Z(G)}\mid=p^4$ ($p$ a prime). If $G$ is not a CA-group, then $\mid \Cent(G) \mid =p^3+p^2+p+2$ (\cite[Theorem 3.5]{amiriF}). Moreover, $\frac{G}{Z(G)}$ is of exponent $p$. 
\end{cor}

\begin{proof}
Since $G$ is not a CA-group therefore $G$ must have a proper non-abelian centralizer. If $C(x)$ is non-abelian for some $x \in G \setminus Z(G)$, then we must have $\mid \frac{G}{C(x)}\mid=p$ and hence $\mid \frac{Z(x)}{Z(G)}\mid=p$.

On the otherhand, if $C(y)$ is abelian for some $y \in G \setminus Z(G)$, then $C(x) \cap C(y)=Z(G)$, for if $z \in (C(x) \cap C(y)) \setminus Z(G)$, then $C(z)=C(x)=C(y)$, which is a contradiction. Therefore we have $\mid \frac{G}{C(y)}\mid=p^3$, noting that in the present scenario,   $C(x) \unlhd G$ and $G=C(x)C(y)$. Hence $\mid \frac{Z(y)}{Z(G)}\mid=p$. Now, the result follows  from Theorem \ref{b1}.
\end{proof}

Following Ito \cite{ito}, a finite group $G$ is said to be of conjugate type $(n, 1)$ if every proper centralizer of $G$ is of index $n$. He proved that a group of  conjugate type $(n, 1)$ is nilpotent and $n=p^a$ for some prime $p$. Moreover, he also proved that a group of  conjugate type $(p^a, 1)$ is a direct product of a $p$-group of the same type and an abelian group. The author in \cite{ish1} classified finite $p$-groups of conjugate type $(p, 1)$ and $(p^2, 1)$ upto isoclinism. In the following result, we calculate the number of element centralizers and $z$-classes of a finite group of conjugate type $(p, 1)$. Given a group $G$, $nacent(G)$ denotes the set of non-abelian centralizers of $G$. For more information about $nacent(G)$ see \cite{nab, con}.

\begin{prop}\label{prop22}
Let $G$ be a finite group such that $ \mid \frac{G}{C(x)} \mid=p$ ($p$ a prime) for all $x \in G \setminus Z(G)$. Then 
\begin{enumerate}
	\item  $\frac{G}{Z(G)}$ is elementary abelian $p$-group of order $p^k$ for some $k$.
	\item  $\mid \Cent(G) \mid =\mid $z$-class\mid=\frac{p^k-1}{p-1}+1$.
	\item  $\mid nacent(G)\mid= 1$ iff $\frac{G}{Z(G)} \cong C_p \times C_p$.
	\item  $\mid nacent(G)\mid= \mid \Cent(G) \mid$ iff $\mid \frac{G}{Z(G)} \mid > p^2$.
\end{enumerate} 
\end{prop}

\begin{proof}
a) In view of Ito \cite{ito}, $G = A \times P$, where $A$ is an abelian group and $P$ is a $p$-group of conjugate type $(p, 1)$. Therefore using Remark \ref{rem1},  $\frac{G}{Z(G)}$ is elementary abelian $p$-group of order $p^k$ for some $k$.\\

b) The result follows from Theorem \ref{b1} and Remark \ref{rem2}, noting that in the present scenario, by \cite[Proposition 1]{mann}, we have $\mid \frac{Z(x)}{Z(G)}\mid=p$ for all $x \in G \setminus Z(G)$.\\

c) If $\frac{G}{Z(G)} \cong C_p \times C_p$, then $G$ is a CA-group and hence $\mid nacent(G)\mid= 1$. Conversely, suppose $\mid nacent(G)\mid= 1$. Then in view of \cite[Proposition 1]{mann}, we have $\frac{G}{Z(G)} \cong C_p \times C_p$.\\

d) In view of \cite[Proposition 1]{mann}, we have $\mid \frac{Z(x)}{Z(G)}\mid=p$ for all $x \in G \setminus Z(G)$. Hence the result follows.
\end{proof}

We now compute the number of distinct centralizers and $z$-classes of an extraspecial $p$-group ($p$ a prime). It is well known that every extraspecial $p$-group has order $p^{2a+1}$ for some positive integer $a$ and for every positive integer $a$, there exists, upto isomorphism, exactly two extraspecial groups of order $p^{2a+1}$.

\begin{prop}\label{prop227}
Let $G$ be an extraspecial $p$-group of order $p^{2a+1}$ for some prime $p$. Then $\mid \Cent(G) \mid =\mid z-class \mid=\frac{p^{2a}-1}{p-1}+1$.
\end{prop}

\begin{proof}
Using \cite[Pp. 5]{lewis}, we have $\mid C(x) \mid= p^{2a}$ for all $x \in G \setminus Z(G)$. Now the result follows from Proposition \ref{prop22}, noting that here we have $\mid Z(G) \mid=p$.
\end{proof}

Our next result concerns about finite groups of conjugate type $(p^2, 1)$.

\begin{prop}\label{prop223}
Let $G$ be a finite group such that $ \mid \frac{G}{C(x)} \mid=p^2$ ($p$ a prime) for all $x \in G \setminus Z(G)$. Then one of the following assertions hold:
\begin{enumerate}
	\item  $\frac{G}{Z(G)}$ is elementary abelian $p$-group.
	\item  $\frac{G}{Z(G)}$ is non-abelian of order $p^3$ ($p$ odd) and of exponent $p$; and $\mid \Cent(G) \mid =p^2+p+2$.
\end{enumerate} 
\end{prop}

\begin{proof}
If $\frac{G}{Z(G)}$ is abelian, then by Remark \ref{rem1}, $\frac{G}{Z(G)}$ is elementary abelian $p$-group. 
Next, suppose $\frac{G}{Z(G)}$ is non-abelian. In view of \cite[Proposition 1 and Proposition 2]{mann}, $\mid \frac{Z(x)}{Z(G)} \mid=p$ for all $x \in G \setminus Z(G)$. Hence using Theorem \ref{b1}, $\frac{G}{Z(G)}$ is of exponent $p$.  Therefore $p$ is odd and consequently,  using \cite[Theorem 4.2]{ish1}, $G$ is isoclinic to a group of order $p^5$ with center of order $p^2$. Hence $\frac{G}{Z(G)}$ is non-abelian of order $p^3$.  Moreover, using Theorem \ref{b1} again, we have $\mid \Cent(G) \mid =p^2+p+2$.
\end{proof}

It may be mentioned here that for the  group G:= Small group (64, 73) in \cite{gap} (see also \cite{ed09}), we have $\frac{G}{Z(G)} \cong C_2 \times C_2 \times C_2$ and $ \mid \frac{G}{C(x)} \mid=4$ for all $x \in G \setminus Z(G)$.\\

The following Proposition generalizes \cite[Theorem 3.8]{amiriF}.

\begin{prop}\label{pp1}
Let $G$ be a finite F-group such that $\mid \frac{G}{Z(G)}\mid=p^k$ ($p$ a prime). If $\mid \frac{Z(x)}{Z(G)} \mid \leq p^2$ for all $x \in G \setminus Z(G)$, then $\mid \Cent(G) \mid=p^{k-1}+p^{k-2}+ \dots+p+2-vp$, where   $v$ is the number of centralizers for which $\mid \frac{Z(x)}{Z(G)}\mid = p^2$.
\end{prop}

\begin{proof}
Since $G$ is a finite F-group, therefore by Remark \ref{rem1}, $Z(x) \cap Z(y)=Z(G)$ for all $x, y \in G \setminus Z(G)$ with $C(x) \neq C(y)$. Let $v$ be the number of centralizers for which $\mid \frac{Z(x)}{Z(G)} \mid = p^2$. Then the centers of these $v$ number of centralizers will contain exactly $v(p^2-1)$ distinct right cosets of $Z(G)$ different from $Z(G)$. On the other-hand the center of each of the remaining proper centralizers will contain exactly $(p-1)$ distinct right cosets of $Z(G)$ other than $Z(G)$. Consequently,
\[
\mid \Cent(G) \mid= \frac{(p^k-1)-v(p^2-1)}{p-1}+v+1=p^{k-1}+p^{k-2}+ \dots+p+2-vp.
\]
\end{proof}

As an immediate corollary we obtain the following result for finite groups of conjugate type $(p^2, 1)$. As we have already mentioned, Ishikawa \cite{ish} proved that I-groups are of class at most $3$.

\begin{cor}\label{cor2}
Let $G$ be a finite group such that $ \mid \frac{G}{C(x)} \mid=p^2$ ($p$ a prime) for all $x \in G \setminus Z(G)$. Then $\mid \Cent(G) \mid=p^{k-1}+p^{k-2}+ \dots+p+2-vp$, where $\mid \frac{G}{Z(G)}\mid=p^k$ and $v$ is the number of centralizers for which $\mid \frac{Z(x)}{Z(G)} \mid=p^2$.
\end{cor}

\begin{proof}
It follows from Proposition \ref{pp1}, noting that in the present scenario, in view of \cite[Proposition 1]{mann},  we have $\mid \frac{Z(x)}{Z(G)} \mid \leq p^2$ for all $x \in G \setminus Z(G)$.
\end{proof}

We have already seen in Proposition \ref{prop223} that if $G$ has class $3$, then $v=0$ and $k=3$;  consequently, 
$\mid \Cent(G) \mid =p^2+p+2$.\\

We now prove the following result which improves \cite[Theorem 3.3]{en09}. A group $G$ is said to be $n$-centralizer if $\mid \Cent(G) \mid=n$.

\begin{prop}\label{kkk}
Let $G$ be a finite $n(=p^2+2)$-centralizer group ($p$ a prime). Then $\mid \frac{G}{C(x)}\mid=p^2$ for all $x \in G \setminus Z(G)$ iff $\frac{G}{Z(G)} \cong C_p \times C_p \times C_p \times C_p$ and $G$ is an F-group.
\end{prop}

\begin{proof}
Suppose $\mid \frac{G}{C(x)}\mid=p^2$ for all $x \in G \setminus Z(G)$. Then $G$ is an F-group.
In view of Proposition \ref{prop223},  $\frac{G}{Z(G)}$ is elementary abelian. Let $X_i=C(x_i), 1\leq i \leq n-1$ where $x_i \in G \setminus Z(G)$. We have $G=\overset{n-1}{\underset{i=1}{\cup}}X_i$ and  $\mid G \mid=\overset{n-1}{\underset{i=2}{\sum}} \mid X_i \mid$. Therefore interchanging $X_i$'s and applying  \cite[Cohn's Theorem]{cohn}, we have $G=X_1X_2$ and $X_1 \cap X_2=Z(G)$. Hence $\frac{G}{Z(G)} \cong C_p \times C_p \times C_p \times C_p$.

Conversely, suppose $\frac{G}{Z(G)} \cong C_p \times C_p \times C_p \times C_p$ and $G$ is a finite F-group. In view of Corollary \ref{cor1}, $G$ is a CA-group.  Now, suppose $\mid \frac{G}{C(x)}\mid=p$ for some $x \in G \setminus Z(G)$. Then in view of Theorem \ref{ff4} and \cite[Theorem 2.3]{baishya}, $\mid \Cent(G) \mid \neq p^2+2$ (noting that in the present scenario $C(x) \unlhd G$). Therefore $\mid \frac{Z(x)}{Z(G)} \mid \leq p^2$ for all $x \in G \setminus Z(G)$ and consequently, using Proposition  \ref{pp1} we have $\mid \frac{G}{C(x)}\mid=p^2$ for all $x \in G \setminus Z(G)$.
\end{proof}

\begin{rem}\label{rem12}
Let $G$ be a finite group and let $H$ be a subgroup of $G$. Then $(G, H)$ is called a Camina pair if $x$ is conjugate in $G$ to $xz$ for all $x \in G \setminus H$. A finite group $G$ is called a Camina group if $(G, G')$ is a Camina pair.

Recall that a $p$-group ($p$ a prime) $G$ is semi-extraspecial, if $G$ satisfies the property for every maximal subgroup $N$ of $Z(G)$ that $\frac{G}{N}$ is an extraspecial group. It is known that every semi-extraspecial group is a special group (\cite[Pp.5]{lewis}) and for such groups $\mid C(x) \mid$ is equal to the index $[G: G']$ for all $x \in G \setminus G'$ (\cite[Theorem 5.5]{lewis}). Following \cite[Theorem 5.2]{lewis}, a group $G$ is a semi-extraspecial $p$-group for some prime $p$ if and only if $G$ is a Camina group of nilpotence class $2$. A group $G$ is said to be ultraspecial if $G$ is semi-extraspecial and $\mid G' \mid= \sqrt{\mid G: G' \mid}$.
It is known that for each prime $p$ there are $p+3$ ultraspecial groups of order $p^6$ and all of the ultraspecial groups of order $p^6$ (including for $p=2$) are isoclinic  (\cite[Pp. 9] {lewis}). 
\end{rem}

In view of the above discussions using Remark \ref{rem2} and Proposition \ref{f1}, we have the following result for an ultraspecial group of order $p^6$, $p$ a prime.

\begin{prop}\label{ultra}
If $G$ is an   ultraspecial group of order $p^6$ for some prime $p$, then $\mid \Cent(G) \mid =\mid z-class \mid=p^2+2$.
\end{prop}

It may be mentioned here that the authors in \cite [Remark 3.15]{rahul1} gave an example of a group to show that the conditions (1) and (2) in \cite [Theorem 3.13]{rahul1} are not sufficient to attain the bound of the Theorem. They  obtained their conclusion using the proof of \cite [Theorem 3.13]{rahul1} without computing $\mid z-class \mid$. However, in view of the above Proposition we can give a larger family of examples of groups (namely, any ultraspecial group of order $p^6$, $p$ a prime) with the precise value of 
$\mid z-class \mid$ which satisfies the conditions (1) and (2) in \cite [Theorem 3.13]{rahul1}
but does not attain the bound of the Theorem. We also want to mention that the group in \cite [Remark 3.15]{rahul1} is an example of an ultraspecial group of order $p^6$, $p$ a prime. \\

As an immediate application to  Proposition \ref{kkk}, we have the following result:

\begin{prop}\label{rem1234}
Let $G$ be a finite group and $p$ a prime. Then $G$ is $n(=p^2+2)$-centralizer with $\mid \frac{G}{C(x)}\mid=p^2$ for all $x \in G \setminus Z(G)$ if and only if $G$ is isoclinic to an ultraspecial group of order $p^6$.
\end{prop}

\begin{proof}
In view of Proposition \ref{prop223}, $\frac{G}{Z(G)}$ is abelian. Therefore $G$ is nilpotent of class $2$ and hence by \cite[Theorem 4.1]{ish1}, Proposition \ref{kkk} and Remark \ref{rem12} we have the result.
\end{proof}

In \cite[Theorem 3.4]{ctc09}, the author proved that if $G$ is a finite $6$-centralizer group,   then $\frac{G}{Z(G)} \cong D_8, A_4, C_2 \times C_2 \times C_2$ or $C_2 \times C_2 \times C_2 \times C_2$. In this connection, we have the following result:

\begin{prop}
A finite group $G$ is $6$-centralizer with $\mid \frac{G}{Z(G)} \mid=16$ if and only if $G$ is isoclinic to an ultraspecial group of order $64$. 
\end{prop}

\begin{proof}
Suppose $\mid \Cent(G) \mid =6$.  Using \cite[Proposition 2.5 (a)]{ed09}, $G$ is a CA-group. Now, suppose $\mid \frac{G}{C(x)}\mid=2$ for some $x \in G \setminus Z(G)$. Then in view of Theorem \ref{ff4} and \cite[Theorem 2.3]{baishya}, we have $\mid \Cent(G) \mid \neq 6$. Consequently, by Proposition \ref{pp1}, we have $\mid \frac{G}{C(x)}\mid=4$ for all  $x \in G \setminus Z(G)$. Therefore  by Proposition \ref{rem1234}, $G$ is isoclinic to an ultraspecial group of order $64$. 
Conversely, suppose  $G$ is isoclinic to an ultraspecial group of order $64$. Then by Remark \ref{rem12}, $\mid G' \mid=4$ and $\mid \frac{G}{C(x)}\mid=4$ for all  $x \in G \setminus Z(G)$. Since we have  $\mid \frac{G}{Z(G)} \mid=16$, therefore  by Proposition \ref{f1}, $\mid \Cent(G) \mid = 6$.
\end{proof}

We now compute the number of centralizers of a finite group with maximal centralizers (maximal among the proper subgroups). It may be mentioned here that Kosvintsev in 1973 \cite{kos} studied and characterised these groups. He proved that in a finite nilpotent group $G$ every centralizer is maximal if and only if $G$ is of the conjugate type $(p, 1)$, $p$ a prime. Moreover, he proved that in a finite non-nilpotent group $G$, the centralizer of every non-central element is a maximal subgroup if and only if $G=MZ(G)$, where $M$ is a biprimary subgroup in $G$ that is a Miller-Moreno group.
Recently, in 2020 the authors \cite{zhao} studied these groups. It seems that the authors are unaware of the paper of Kosvintsev \cite{kos}. However, they have given a characterization of such non-nilpotent group $G$ in terms of $\frac{G}{Z(G)}$ by proving that if $G$ is a finite non-nilpotent group of such type, then $\frac{G}{Z(G)}$ is either abelian or $\frac{G}{Z(G)}=\frac{P}{Z(G)} \rtimes \frac{Q}{Z(G)}$ is  a minimal non-abelian group (Miller and Moreno analyzed minimal non-abelian groups in \cite{mm}) with $\mid \frac{P}{Z(G)} \mid =p^a$ and $\mid \frac{Q}{Z(G)} \mid =q$, where $p$ and $q$ are primes. In this connection, we notice that there is a minor error in the result. In this case $\frac{G}{Z(G)}$ cannot be abelian.

\begin{prop}
Let $G$ be a finite group in which every proper centralizer is maximal (maximal among all proper subgroups). 
\begin{enumerate}
	\item  If $G$ is nilpotent, then $\mid \frac{G}{Z(G)} \mid=p^k$ ($p$ a prime) and $\mid \Cent(G) \mid =\mid $z$-class\mid=\frac{p^k-1}{p-1}+1$. Moreover,  $\mid nacent(G)\mid= 1$ iff $\frac{G}{Z(G)} \cong C_p \times C_p$ and $\mid nacent(G)\mid= \mid \Cent(G) \mid$ iff $\mid \frac{G}{Z(G)} \mid > p^2$.
	\item  If $G$ is non-nilpotent, then  $\mid \Cent(G) \mid = \mid \Cent(\frac{G}{Z(G)}) \mid =p^a+2$, where $p^a$ is the order of the Sylow $p$-subgroup of $\frac{G}{Z(G)}$. Moreover, $G$ is a CA-group.
\end{enumerate} 
\end{prop}

\begin{proof}
a) Since $G$ is nilpotent, therefore  in view of \cite{kos}, we have  $\mid \frac{G}{C(x)}\mid=p$ ($p$ a prime) for all $x \in G \setminus Z(G)$. Now, the result follows using Proposition \ref{prop22}.\\

b) In view of \cite[Theorem A]{zhao}, we have $\frac{G}{Z(G)}=\frac{P}{Z(G)} \rtimes \frac{Q}{Z(G)}$ is  a minimal non-abelian group with $\mid \frac{P}{Z(G)} \mid =p^a$ and $\mid \frac{Q}{Z(G)} \mid =q$, where $p$ and $q$ are primes. Moreover, by \cite[Aufgaben III, 5.14]{hupert}, $\frac{G}{Z(G)}$ has trivial center. In the present scenario, we have  $\frac{C(x)}{Z(G)}=C(xZ(G))$ for any $x \in G \setminus Z(G)$ (because $C(x)$ is a maximal subgroup). Hence $\mid \Cent(G) \mid = \mid \Cent(\frac{G}{Z(G)}) \mid =p^a+2$ and $G$ is a CA-group, noting that in the present scenario $G$ is an F-group.
\end{proof}

The authors in \cite[Main Theorem 1.1]{zclass}, gave a necessary and sufficient condition for a finite $p$-group ($p$ a prime) of type $(n, 1)$ to attain the maximal number of $z$-classes.  In the following Theorem we extend this result as follows:

\begin{thm}\label{prop35}
Let $G$ be a finite F-group with  $\mid \frac{G}{Z(G)}\mid=p^k$, where $p$ is a prime. Then 
$G$ has $\frac{p^k-1}{p-1}+1$ $z$-classes if and only if 
\begin{enumerate}
	\item  $\frac{G}{Z(G)}$ is elementary abelian. 
	\item  $Z(x)= \langle x, Z(G)\rangle$ for all $x \in G \setminus Z(G)$.
\end{enumerate} 
\end{thm}

\begin{proof}
Let $G$ be a finite F-group such that  $\mid \frac{G}{Z(G)}\mid=p^k$ and  
$G$ has $\frac{p^k-1}{p-1}+1$ $z$-classes. By \cite[Lemma 3.1]{rahul1}, $G$ is isoclinic to a finite $p$-group $H$ and by  \cite[Theorem 2.2]{rahul1}, $H$ has $\frac{p^k-1}{p-1}+1$ $z$-classes. In the present scenario, by \cite[Theorem 3.13]{rahul1},  $\frac{H}{Z(H)}$ is elementary abelian and consequently, $\frac{G}{Z(G)}$ is elementary abelian. Hence $C(x) \unlhd G$ for all $x \in G $ and therefore, by Remark \ref{rem2},  $\mid \Cent(G) \mid =\frac{p^k-1}{p-1}+1$. Now, the result follows from Theorem \ref{b1}.

Conversely, suppose (a) and (b) holds. Then $C(x) \unlhd G$ for all $x \in G $ and consequently, by Remark \ref{rem2} and  Theorem \ref{b1} we have $\mid z-class \mid =\frac{p^k-1}{p-1}+1$.
\end{proof}

We conclude the paper with the following result:

\begin{prop}
Let $G$ be a finite group such that $\mid G' \mid=p$ ($p$  a prime) and $G' \subseteq Z(G)$. Then  $G$ is isoclinic to an extraspecial $p$-group of order $p^{2a+1}$ and $\mid \Cent(G) \mid=\mid z-class \mid =\frac{p^{2a}-1}{p-1}+1$.
\end{prop}

\begin{proof}
Since $G' \subseteq Z(G)$ therefore $G=A \times H$, where $A$ is an abelian subgroup and $H$ is the Sylow $p$-subgroup of $G$ with $\mid H' \mid=p$. Consequently, $G$ is isoclinic to $H$. In the present scenario, by Theorem \ref{ff1} and Theorem \ref{ff3}, $H$ is isoclinic to a finite $p$-group $H_1$ with $Z(H_1) \subseteq H_1'$. Therefore we have $\mid Z(H_1) \mid=\mid H_1' \mid=p$. It now follows that $H_1'$ and $\frac{H_1}{Z(H_1)}$ have same exponent (by \cite[Lemma 9, p. 77]{kiture5}). Thus $H_1$ is an extraspecial $p$-group of order $p^{2a+1}$ and $G$ is isoclinic to $H_1$. Moreover,  by Proposition \ref{prop227} and Theorem \ref{ff2} we have $\mid \Cent(G) \mid=\mid z-class \mid =\frac{p^{2a}-1}{p-1}+1$.

\end{proof}

\section*{Acknowledgment}

I would like to thank Prof. Mohammad Zarrin for his valuable suggestions and comments on the earlier draft of the paper.


\begin{thebibliography}{33}

\bibitem{ed09}
A. Abdollahi, S. M. J. Amiri and A. M. Hassanabadi,  {\em Groups with specific number of centralizers}, Houst. J. Math., \textbf{33} (1) (2007), 43--57.


\bibitem{kiture5}
J. L. Alperin and R. B. Bell,  {\em Groups and representations}, Graduate Texts in Mathematics,  \textbf{162} (Spinger-Verlag, NY, 19951).


\bibitem{amiriF}
 S. M. J. Amiri, H. Madadi and H. Rostami,  {\em On F-Groups with the central factor of order $p^4$},  Math. Slovaca, \textbf{67} (5)  (2017), 1147--1154.



\bibitem{amiriF1}
 S. M. J. Amiri, H. Madadi and H. Rostami,  {\em Finite groups with certain number of centralizers},  Third Biennial International Group Theory Conference, Mashhad 2015.

\bibitem{nab}
 S. M. J. Amiri and H. Rostami,  {\em Groups with a few non-abelian centralizers},  Publ. Math. Debrecen, \textbf{87} (3/4)  (2015), 429--437.





\bibitem{zclass}
 S. Arora and K. Gongopadhyay,  {\em $z$-Classes in finite group of conjugate type $(n, 1)$},  Proc. Indian Acad. Sci. (Math. Sci.), \textbf{128} (31)  (2018), 1--7.



\bibitem{en09}
 A. R. Ashrafi, {\em On finite groups with a given number of centralizers}, Algebra Colloq., \textbf{7} (2) (2000), 139--146.
 
 
 \bibitem{ctc09}
A. R. Ashrafi, {\em Counting the centralizers of some finite groups}, Korean  J. Comput. Appl. Math., \textbf{7} (1) (2000), 115--124.
 

\bibitem{baishya}
S. J. Baishya, {\em On  finite groups with specific number of centralizers}, Int. Elect. J. Algebra, \textbf{13} (2013) 53--62.




\bibitem{ctc092}
S. M. Belcastro and G. J. Sherman, {\em Counting centralizers in finite groups}, Math. Magazine, \textbf{67} (5) (1994), 366--374.



\bibitem{zumud}
Y. G. Berkovich and   E. M. Zhmud$^{\prime}$,   \textit{Characters of Finite Groups}. Part 1, Transl. Math. Monographs {\bf 172}, Amer. Math. Soc., Providence, RI, 1998. 


\bibitem{brough}
J. Brough, {\em Central intersections of element centralizers}, Asian-European J. Math.,  \textbf{11} (5) (2018), (11 pages).


\bibitem{cohn}
J. H. E. Cohn, {\em On $n$-sum groups}, Math. Scand, \textbf{75} (1994), 44--58.



\bibitem{gon1}
K. Gongopadhyay, {\em The $z$-classes of quaternionic hyperbolic isometries}, J. Group Theory,  \textbf{16} (6) (2013), 941--964.


\bibitem{gon2}
K. Gongopadhyay and R. S. Kulkarni, {\em  $z$-classes of isometries of the hyperbolic space}, Conform. Geom. Dyn.,  \textbf{13}  (2009), 91--109.


\bibitem{gon3}
K. Gongopadhyay and R. S. Kulkarni, {\em  The $z$-classes of isometries}, J. Indin Math. Soc.,  \textbf{81} (3--4) (2014), 245--258.


\bibitem{hall}
P. Hall, {\em The classification of prime-power groups},  J. Reine Angew Math.,  \textbf{182} (1940), 130--141.


\bibitem{hupert}
B. Huppert, {\em Endliche Gruppen, I},  (Springer-Verlag, Berlin 1967).





\bibitem{ish}
K. Ishikawa, {\em On finite $p$-groups which have only two conjugacy lengths}, Israel J. Math.,  \textbf{129} (2002), 119--123.


\bibitem{ish1}
K. Ishikawa, {\em Finite $p$-groups upto isoclinism, Which have only two conjugacy lengths},  J. Algebra,  \textbf{220} (1999), 333--345.


\bibitem{ito}
N. Ito, {\em On finite groups with given conjugate type, I}, Nagoya J. Math.,  \textbf{6} (1953), 17--28.

\bibitem{rahul2}
R. D. Kitture, {\em Groups with finitely many centralizers}, Bull. Allahabad Math. Soc., \textbf{30}  (2015), 29--37.

\bibitem{rahul3}
R. D. Kitture, {\em $z$-classes in finite $p$-groups}, Ph.D thesis.  (2014) (University of Pune).



\bibitem{con}
K. Khoramshahi and M. Zarrin, {\em Groups with the same number of centralizers}, J. Algebra Appl. https://doi.org/10.1142/S0219498821500122.




\bibitem{kos}
L. F. Kosvintsev, {\em Finite groups with maximal element centralizers}, Mathematical Notes of the Academy of Sciences of USSR \textbf{13}  (1973), 349--350.




\bibitem{kul}
R. S. Kulkarni, {\em Dynamics of linear and affine maps}, Asian J. Math.,  \textbf{12} (3) (2008), 321--344.


\bibitem{kul1}
R. S. Kulkarni, {\em Dynamical types and conjugacy classes of centralizers in groups }, J. Ramanujan Math. Soc.,  \textbf{22} (1) (2007), 35--36.




\bibitem{rahul1}
R. Kulkarni, R. D. Kitture and V. S. Yadav, {\em $z$-classes in groups}, J. Algebra Appl., \textbf{15} (2016), 01--16.





\bibitem{rahul4}
V. S. Jadhav and R. D. Kitture, {\em $z$-classes in finite $p$-groups of order $\leq p^5$}, Bull. Allahabad Math. Soc., \textbf{29} (2) (2014), 173--194.








\bibitem{mm}
G. A. Miller and H. C. Moreno, {\em Non-abelian groups in which every subgroup is abelian}, Trans. Amer. Math. Soc., \textbf{4}  (1903), 398--404.





\bibitem{lewis}
M. L. Lewis, {\em Semi extraspecial groups}, Advances in algebra, Springer Proc. Math. Stat., \textbf{277}, Springer, Cham,  (2019), 219--237.


\bibitem{mann}
A. Mann, {\em Extreme elements op finite $p$-groups}, Rend. Sem. Mat. Univ. Padova,  \textbf{83} (1990), 45--54.

\bibitem{reb}
J. Rebmann, {\em F-gruppen}, Arch. Math.,  \textbf{22} (1971), 225--230.


\bibitem{part}
L. J. Taghvasani and M. Zarrin, {\em Criteria for nilpotency of groups via partitions}, Math.  Nachr.,  \textbf{291} (17--18) (2018), 2585--2589.



\bibitem{tom}
M. J. Tomkinson, {\em Groups covered by finitely many cosets or subgroups}, Comm.  Algebra,  \textbf{15} (4)  (1987), 854--859.

  

\bibitem{zhao}
X. Zhao, R. Chen, X. Guo {\em Groups in which the centralizer of any non-central element is maximal}, J. Group Theory,  \textbf{23} (2020), 871--878.





\bibitem{zappa}
G. Zappa, {\em Partitions and other coverings of finite groups}, Illinois. J. Math., \textbf{47} (1/2) (2003), 571--580.

\bibitem{zarrin094}
M. Zarrin, {\em On element centralizers in finite groups}, Arch. Math.,  \textbf{93} (2009), 497--503.


\bibitem{jaa4}
M. Zarrin, {\em Derived length and centralizers of groups}, J. Algebra Appl., \textbf{14} (8) (2015), 01--04.





\bibitem{non}
M. Zarrin, {\em On non-commuting sets and centralizers in infinite groups}, Bull. Aust. Math. Soc., \textbf{93} (2016), 42--46.




\bibitem{gap}
The GAP Group, GAP- Groups, Algorithoms and Programming, Version 4.6.4, http://www.gap-system.org, 2013. 




\end{thebibliography}
\end{document}